\documentclass[11pt,a4paper]{article} 
\usepackage{amsmath,amssymb,latexsym,mathrsfs,amsthm,amscd}
\input amssym.def 
\def\R{\Bbb R} \def\C{\Bbb C}  \def\O{\Bbb O} 
  \def\P{\Bbb P}   
 \def\G{\Bbb G} \def\Z{\Bbb Z}

\newtheorem{theo}{Theorem}[section] 
\newtheorem{lem}[theo]{Lemma} 
\newtheorem{pro}[theo]{Proposition} 
\newtheorem{rem}[theo]{Remark} 
\newtheorem{cor}[theo]{Corollary}

\title{The Degree of an Eight-Dimensional Real Quadratic Division Algebra is 1, 3, or 5}  
\author{Ernst Dieterich and Ryszard Rubinsztein} 

\begin{document} 
\date{} 
\maketitle 

\begin{abstract}
\noindent 
A celebrated theorem of Hopf, Bott, Milnor, and Kervaire \cite{Ho40},\cite{BM58},\cite{Ke58} states that every finite-dimensional real division algebra has dimension 
1, 2, 4, or 8. While the real division algebras of dimension 1 or 2 and the real quadratic division algebras of dimension 4 have been classified 
\cite{Di05},\cite{Di98},\cite{DO02}, the problem of classifying all 8-dimensional real quadratic division algebras is still open. We contribute to a solution of that 
problem by proving that every 8-dimensional real quadratic division algebra has degree 1, 3, or 5. This statement is sharp. It was conjectured in \cite{DF06}.   
\end{abstract} 

\\[1ex]
  
\section{Introduction}

Let $A$ be an algebra over a field $k$, i.e.~a vector space over $k$ equipped with a $k$-bilinear multiplication $A \times A \to A,\ (x,y) \mapsto xy$. Every element 
$a \in A$ determines $k$-linear operators $L_a: A \to A,\ x \mapsto ax$ and $R_a: A \to A$, \mbox{$x \mapsto xa$}. A {\it division algebra} over $k$ is a non-zero $k$-algebra 
$A$ such that $L_a$ and $R_a$ are bijective for all $a \in A \setminus \{0\}$. A {\it quadratic algebra} over $k$ is a non-zero $k$-algebra $A$ with unity 1, such that 
for all $x \in A$ the sequence $1,x,x^2$ is $k$-linearly dependent.

Here we focus on quadratic division algebras which are real \mbox{(i.e.~$k = \R$)} and finite-dimensional. They form a category $\mathscr{D}^q$ whose morphisms \linebreak 
\mbox{$f: A \to B$} are non-zero linear maps satisfying $f(xy) = f(x)f(y)$ for all $x,y \in A$. For any positive integer $n$, the class of all $n$-dimensional objects in 
$\mathscr{D}^q$ forms a full subcategory $\mathscr{D}_n^q$ of $\mathscr{D}^q$. The 
(1,2,4,8)-theorem for finite-dimensional real division algebras implies that $\mathscr{D}^q = \mathscr{D}_1^q \cup \mathscr{D}_2^q \cup \mathscr{D}_4^q \cup 
\mathscr{D}_8^q$. It is well-known (see e.g.~\cite{DO02}) that both $\mathscr{D}_1^q$ and $\mathscr{D}_2^q$ consist of one isoclass only, represented by $\R$ and $\C$ 
respectively. The category $\mathscr{D}_4^q$ is no longer trivial, but still its structure is well-understood and its objects are classified by a 9-parameter family 
\cite{Di98},\cite{Di00},\cite{DO02}. In contrast, the category $\mathscr{D}_8^q$ seems to be much more difficult to approach. The problem of understanding its structure 
appears to have a very hard core, which still is far from a finishing solution. Among the insight so far obtained, the {\it degree} of an 8-dimensional real quadratic 
division algebra is notable. Following \cite{DF06} it is a natural number $\deg(A)$ associated with any $A \in \mathscr{D}_8^q$, which is invariant under isomorphisms and 
satisfies the estimate $1 \le \deg(A) \le 5$. Examples of algebras $A \in \mathscr{D}_8^q$ having degree 1, 3, and 5 respectively are also given in 
\cite{DF06}. For every $d \in \{ 1,\ldots,5 \}$, the class of all algebras in $\mathscr{D}_8^q$ having degree $d$ forms a full subcategory $\mathscr{D}_8^{qd}$ of 
$\mathscr{D}_8^q$. 

In the present article we resume the investigation of $\deg(A)$. In section 2 we apply topological arguments to prove that $\deg(A)$ always is odd. It follows that the 
category $\mathscr{D}_8^q$ decomposes into three non-empty blocks $\mathscr{D}_8^{q1}, \mathscr{D}_8^{q3}$, and $\mathscr{D}_8^{q5}$. In section 3 we summarize our structural 
insight into these three blocks, concluding that the problem of understanding the structure of $\mathscr{D}_8^q$ has been reduced to the problem of understanding the 
structure of the category of all 7-dimensional {\it dissident algebras} having degree 3 or 5. 

\section{Main result}

Towards our main result we need to recall a few facts related to the notion of a dissident map. Let $V$ be a vector space over a field $k$. Following \cite{DL03}, a 
{\it dissident map} on $V$ is a $k$-linear map $\eta: V \wedge V \to V$ such that $v, w, \eta(v \wedge w)$ are $k$-linearly independent whenever $v,w$ are. It is well-known 
that a finite-dimensional real vector space $V$ admits a dissident map if and only if $\dim(V) \in \{0,1,3,7\}$ \cite[Proposition 7]{Di99}.

Let $\eta: \R^7 \wedge \R^7 \to \R^7$ be a dissident map. We equip $\R^7$ with the standard scalar product $\R^7 \times \R^7 \to \R,\ (v,w) \mapsto v^tw$. For every $v \in 
\R^7 \setminus \{0\}$, the subspace $\eta(v \wedge v^\perp) = \{ \eta(v \wedge w)\ |\ w \in v^\perp \}$ in $\R^7$ is a hyperplane, which only depends on the line $[v]$ 
spanned by $v$. Thus $\eta$ induces a map 
\[ \eta_\P: \P(\R^7) \to \P(\R^7),\ \eta_\P([v]) = \left( \eta(v \wedge v^\perp) \right)^\perp, \]  
which actually is bijective \cite[Proposition 2.2]{DL03}. Following \cite{DF06}, a {\it lifting} of $\eta_\P$ is a map $\Phi: \R^7 \to \R^7,\ \Phi(v) = (\varphi_1(v),\ldots,
\varphi_7(v))$, satisfying the following three conditions: (a) all component maps $\varphi_1,\ldots,\varphi_7$ are homogeneous real polynomials of common degree $d \ge 1$; 
(b) if $v \in \R^7 \setminus \{0\}$, then $\Phi(v) \not= 0$ and $[\Phi(v)] = \eta_\P([v])$; (c) the polynomials $\varphi_1,\ldots,\varphi_7$ are relatively prime.

According to \cite[Theorem 2.4]{DF06} a lifting $\Phi$ of $\eta_\P$ exists, is unique up to non-zero real multiples, and satisfies $1 \le d \le 5$. It is therefore justified 
to define the {\it degree} of a dissident map $\eta$ on $\R^7$ by $\deg(\eta) := \deg(\Phi) := d$. 
\begin{pro}
The degree of any dissident map on $\R^7$ is odd.
\end{pro}

\begin{proof}
Our proof of Proposition 2.1 is topological in nature and, in particular, uses the fundamental group of the projective space $\, \P (\R ^7 )\,$.

Let $\, \pi :\R^7 \setminus \{0\} \rightarrow  \P (\R ^7 )\,$ be the quotient map mapping a vector $\, v\in \R^7 \setminus \{0\}\,$ to the line $\, [v]\in  \P (\R ^7 )\,$ 
spanned by $\, v\,$. We equip $\,\R^7 \setminus \{0\}\,$ with its standard Euclidean topology and $\, \P (\R ^7 )\,$ with the quotient topology.
For a point $\, z\in  \P (\R ^7 )\,$ let $\, \pi _1(  \P (\R ^7 ), z)\,$ be the fundamental group of  $\, \P (\R ^7 )\,$ based at $\, z\,$. We recall that 
\begin{equation*}
 \pi _1(  \P (\R ^7 ), z) = \Z _2 \,\,\, ,
\end{equation*}  
$\, \Z _2 = \{ 0, 1 \}\,$ being the group of integers modulo $2$ \cite[Section III.5]{B1}.

The map  $\, \pi :\R^7 \setminus \{0\} \rightarrow  \P (\R ^7 )\,$ is a locally trivial fibration with the fibre homeomorphic to $\,\R \setminus \{0\}\,$. We choose a point $\, z\in  \P (\R ^7 )\,$ and a vector $\, \tilde{z} \in \R^7 \setminus \{0\}\,$ such that $\, \pi (\tilde{z})=z\,$. Then every vector $\, w\,$ belonging to the fibre $\, \pi^{-1}(z)\,$ is of the form $\, w=q \tilde{z}\,$ with $\, q\in \R, \, q\ne 0.\,$

Denote by $\, I\,$ the unit interval $\,[ 0, 1 ]\,$.

Let $\, \sigma :I\rightarrow  \P (\R ^7 )\,$ be a loop in  $\, \P (\R ^7 )\,$ at the point $\, z\,$ i.e. $\,\sigma \,$ is a continuous mapping such that $\, \sigma (0) = \sigma (1)= z\,$. Since  $\, \pi :\R^7 \setminus \{0\} \rightarrow  \P (\R ^7 )\,$ is a fibration, it follows from the Homotopy Lifting Theorem, \cite[ Theorem VII.6.4]{B1}, that the loop  $\, \sigma :I\rightarrow  \P (\R ^7 )\,$ can be lifted to a continuous mapping  $\, \tilde{\sigma} :I\rightarrow \R^7 \setminus \{0\} \,$ such that $\,\pi \circ  \tilde{\sigma}= \sigma\,$ and $\, \tilde{\sigma}(0)=  \tilde{z}.\,$ (Observe that, in general, $\, \tilde{\sigma}\,$ will not be a loop but just a path.) Since $\, \sigma (1)=z,\,$ it follows that $\,  \tilde{\sigma}(1)\in  \pi^{-1}(z)\,$ and, hence, that  $\,  \tilde{\sigma}(1) =  q \tilde{z}\,$ for some $\, q=q(\sigma )\in \R, \, q\ne 0.\,$ (Actually, the real number $q$ depends not only on the loop $\, \sigma\,$ but also on the choice of the lifting $\,\tilde{\sigma}\,$ which is not unique.) The next lemma is rather obvious. We include a proof for convenience of the reader.

\begin{lem} The loop  $\, \sigma :I\rightarrow  \P (\R ^7 )\,$ represents the trivial element in $\, \pi _1(  \P (\R ^7 ), z) = \Z _2\,$ if and only if $\, q>0\,$. 
\end{lem}

\begin{proof} ($\, \Leftarrow\,$) Suppose that $\, q>0\,$. Then there is a line segment $\, \tilde{\tau}\,$  in $\, \pi ^{-1}(z)\,$ from $\, q\tilde{z}\,$ to $\, \tilde{z}\,$ consisting of points of the form $\, s\tilde{z}\,$ with $\, s\,$ between $1$ and $q$. The path product $\,\overline{\sigma} = \tilde{\sigma} * \tilde{\tau}\,$ is now a loop in $\, \R^7\setminus \{0\}\,$ starting and ending at $\, \tilde{z}\,$. The space $\,\R^7\setminus \{0\}\,$ is homotopy equivalent to the $6$-dimensional sphere $\, S^6\,$ and hence the fundamental group $\, \pi_1(\R^7\setminus \{0\}, \tilde{z} )=0\,$ is trivial. Therefore the loop   $\,\overline{\sigma} = \tilde{\sigma} * \tilde{\tau}\,$ is null-homotopic in  $\,\R^7\setminus \{0\}\,$. It follows that the loop $\, \pi\circ\overline{\sigma} = (\pi \circ \tilde{\sigma})* (\pi\circ \tilde{\tau})= \sigma * (\pi\circ \tilde{\tau})\,$ is null-homotopic in $\, \P (\R ^7 ).\,$  Since $\, \pi\circ \tilde{\tau}\,$  is a constant loop, the loops  $\,\sigma * (\pi\circ \tilde{\tau})\,$ and   $\,\sigma \,$ are homotopic. Therefore the loop $\, \sigma\,$ is null-homotopic in  $\, \P (\R ^7 )\,$ and represents the trivial element in   $\, \pi _1(  \P (\R ^7 ), z).\,$ 

($\, \Rightarrow\,$) Suppose that the loop  $\, \sigma :I\rightarrow  \P (\R ^7 )\,$ represents the trivial element in $\, \pi _1(  \P (\R ^7 ), z).\,$ Thus there exists a continuous mapping (homotopy)\break $\, H: I\times I \rightarrow  \P (\R ^7 )\,$ such that $\, H(t,0)=\sigma (t)\,$ and $\, H(0,s)=H(1,s)=H(t,1)=z\,$ for all $t, s \in I\,$. Again, according to the Homotopy Lifting Theorem there exists a continuous mapping (lift) $\,\tilde{H}:I\times I \rightarrow \R^7\setminus \{0\}\,$ such that $\,\pi \circ\tilde{H}= H\,$ and $\, \tilde{H}(t,0)=  \tilde{\sigma}(t), \,\tilde{H}(0,s)=\tilde{z}\,$ and $\, \tilde{H}(1,s)=\tilde{\sigma}(1)\,$ for all  $t, s \in I\,$.  It follows, in particular, that $\, \tilde{H}(t,1)\in \pi ^{-1}(z)\,$ for all $\, t\in I\,$ and that $\, \tilde{H}(0,1)=\tilde{z}\,$ while   $\, \tilde{H}(1,1)=\tilde{\sigma}(1)= q\tilde{z}\,$. In other words,  $\, \tilde{H}(t,1), \, t\in I,\,$ is an arc in $\, \pi ^{-1}(z)\,$ from $\,\tilde{z}\,$ to $\,\tilde{\sigma}(1)= q\tilde{z}\,$. As $\, \pi ^{-1}(z)\,$ consists of points of the form $\, r \tilde{z}\,$ with $\, r\in \R\setminus \{0\}\,$ it follows that $\, q>0\,$.

That completes the proof of Lemma 2.1.  
\end{proof}

The diagram
\begin{displaymath}
\begin{CD}
\R^7 \setminus \{0\}  @>{\Phi}>>   \R^7 \setminus \{0\}\\
@VV{\pi}V                                     @VV{\pi}V \\ 
 \P (\R ^7 )           @>{\eta_{\P}}>>         \P (\R ^7 ) 
\end{CD}
\end{displaymath}
\noindent
commutes. Since the components  $\, \varphi _1, ... , \varphi _7\,$ of $\, \Phi\,$ are polynomials, the map $\, \Phi : \R^7 \setminus \{0\}\rightarrow \R^7 \setminus \{0\}\,$ is continuous. It follows that the map \break $\,\eta_{\P}:\P (\R ^7 )\rightarrow \P (\R ^7 ) \,$ is also continuous. The map $\,\eta_{\P}\,$ is bijective \cite[Proposition 2.2]{DL03} and the projective space  $\,\P (\R ^7 )\,$ is compact and Hausdorff. Therefore  $\,\eta_{\P}\,$ is a homeomorphism. Let us choose a point $\, z_0\in \P (\R ^7 )  \,$  and let us denote by $\, z_1\,$ the point $\, {\eta_{\P}}(z_0)\in \P (\R ^7 )  \,$. The homeomorphism  $\,\eta_{\P}\,$ induces a group isomorphism
\begin{equation*}
 {\eta_{\P}}_*: \pi _1( \P (\R ^7 ), z_0 ) \xrightarrow{\cong}  \pi _1( \P (\R ^7 ), z_1 )\,\, .
\end{equation*}  

Let us choose a point $\, \tilde{z}_0 \in \pi ^{-1}(z_0) \subset \R^7 \setminus \{0\}\,$. Let us denote by  $\, \tilde{z}_1\,$ the point  $\, \Phi( \tilde{z}_0)\in \R^7 \setminus \{0\} \,$. Then   $\, \tilde{z}_1= \Phi( \tilde{z}_0)\in\pi ^{-1}(z_1)\,$.
 
Since the space $\, \R^7 \setminus \{0\}\,$ is path-connected, we can find  a path $\, \tilde{\alpha}:I\rightarrow  \R^7 \setminus \{0\}\,$ such that $\, \tilde{\alpha}(0)=  \tilde{z}_0\,$ and $\, \tilde{\alpha}(1)= - \tilde{z}_0\,$. Then the composition $\, \alpha = \pi \circ  \tilde{\alpha}:I\rightarrow \P (\R ^7 )  \,$ is a loop in $\, \P (\R ^7 )  \,$ which starts and ends at $\, z_0= \pi ( \tilde{z}_0) =  \pi ( -\tilde{z}_0).\,$ The path $\, \tilde{\alpha}\,$ is a lift of the loop $\, \alpha\,$ to $\, \R^7 \setminus \{0\}\,$ which starts at $\, \tilde{z}_0\,$ and ends at  $\, -\tilde{z}_0\,$. Thus, according to Lemma 2.1, the loop $\, \alpha\,$ represents the nontrivial element of $\,  \pi _1( \P (\R ^7 ), z_0 )\,$.

Let us now consider the path $\, \tilde{\beta}=\Phi \circ \tilde{\alpha}:I\rightarrow \R^7 \setminus \{0\}\,$ in  $\, \R^7 \setminus \{0\}\,$ starting at $\,  \tilde{z}_1= \Phi( \tilde{z}_0)\in\pi ^{-1}(z_1) \,$ and ending at $\, \tilde{z}_2 =  \Phi( -\tilde{z}_0)\in \pi ^{-1}(z_1) \,$. If $\, d=\text{deg}(\Phi)\,$ is the degree of $\, \Phi\,$ then, by the definition of the degree, $\,\tilde{z}_2 =  \Phi( -\tilde{z}_0)= (-1)^d  \Phi( \tilde{z}_0)=(-1)^d  \tilde{z}_1\,$.

The composition $\, \beta= \pi \circ  \tilde{\beta}: I\rightarrow  \P (\R ^7) \,$ is a loop in $\,  \P (\R ^7) \,$  which starts and ends at the point $\, z_1=\pi ( \tilde{z}_1)=\pi ( \tilde{z}_2) \,$. Moreover,
\begin{equation*}
\beta= \pi \circ  \tilde{\beta}=  \pi \circ \Phi \circ \tilde{\alpha}= \eta_{\P}\circ \pi \circ \tilde{\alpha}= \eta_{\P}\circ\alpha\,\,.
\end{equation*}
Thus the homotopy class $\, [ \beta ]\,$ of the loop $\, \beta \,$  in $\, \pi _1( \P (\R ^7 ), z_1 )\,$ is equal to $\, [  \eta_{\P}\circ\alpha ] = {\eta_{\P}}_* [ \alpha ] \,$. As the homotopy class $\, [ \alpha ]\,$ of $\, \alpha \,$  in $\, \pi _1( \P (\R ^7 ), z_0 )\,$ was non-trivial  and $\, {\eta_{\P}}_*: \pi _1( \P (\R ^7 ), z_0 ) \rightarrow  \pi _1( \P (\R ^7 ), z_1 ) \,$  was an isomorphism, it follows that   $\, \beta \,$ represents the non-trivial element of $\, \pi _1( \P (\R ^7 ), z_1 ) \,$. 

On the other hand the path $\, \tilde{\beta}:I\rightarrow \R^7 \setminus \{0\}\,$ is a lifting of  the loop $\, \beta \,$ to $\, \R^7 \setminus \{0\}\,$ which starts at the point $\, \tilde{z}_1\in \R^7 \setminus \{0\}\,$ and ends at the point $\,  \tilde{z}_2 =(-1)^d  \tilde{z}_1\,$. Since  $\, \beta \,$ represents the non-trivial element of $\, \pi _1( \P (\R ^7 ), z_1 ) \,$, it follows from Lemma 2.1 that $\, (-1)^d <0\,$ and, thus, the degree $\, d=\text{deg}(\Phi)\,$ is odd.

That completes the proof of Proposition 2.1.     
\end{proof}

\begin{rem}
The space $\, \R^7 \setminus \{0\}\,$ is homotopy equivalent to the $6$-dimen- sional sphere $ \, S^6\,$. Its homology group $\, H_6(\R^7 \setminus \{0\}, \Z )\,$ with coefficients in $\, \Z\,$ in dimension $6$ is isomorphic to the group of integers,  $\, H_6(\R^7 \setminus \{0\}, \Z )\cong \Z\,$. 
The continuous mapping $\, \Phi: \R^7 \setminus \{0\}\rightarrow \R^7 \setminus \{0\}\,$ induces a group homomorphism $\, \Phi_*: H_6(\R^7 \setminus \{0\}, \Z ) \rightarrow  H_6(\R^7 \setminus \{0\}, \Z )\,$ which is given by multiplication by an integer usually called the (topological) degree of the map $\, \Phi\,$. Since the map $\, \eta_{\P}:  \P (\R ^7 )\rightarrow  \P (\R ^7 )\,$ is a homeomorphism and the algebraic degree $\, d=\text{deg}(\Phi)\,$ in the sense of this paper is odd, it is rather easy to see that also the mapping $\, \Phi: \R^7 \setminus \{0\}\rightarrow \R^7 \setminus \{0\}\,$ is a homeomorphism. It follows that $\, \Phi_*: H_6(\R^7 \setminus \{0\}, \Z ) \rightarrow  H_6(\R^7 \setminus \{0\}, \Z )\,$ is an isomorphism and, hence, its topological degree can only be equal to $\, \pm 1\,$. Thus the topological degree of the map  $\, \Phi\,$ and its algebraic degree $\,d=\text{deg}(\Phi)\,$ in the sense of {\rm\cite{DF06}} and of the present paper are different notions, at least when the algebraic degree  $\,d=\text{deg}(\Phi)$ is equal to $3$ or $5$. 
\end{rem}
\noindent
Now let $A$ be an 8-dimensional real quadratic division algebra. Since $A$ is a quadratic algebra, Frobenius's lemma \cite{Fr78},\cite{KR92} applies. It asserts that the set 
$V = \{ v \in A \setminus \R1\ |\ v^2 \in \R1 \} \cup \{0\}$ of all purely imaginary elements in $A$ is a hyperplane in $A$, such that $A = \R1 \oplus V$. This Frobenius 
decomposition of $A$ gives rise to the $\R$-linear maps $\varrho: A \to \R$ and $\iota: A \to V$ such that $a = \varrho(a)1 + \iota(a)$ for all $a \in A$. The induced algebra 
structure on $V$, i.e.~the bilinear map $\eta: V \times V \to V,\ 
\eta(v,w) = \iota(vw)$, is anticommutative. Therefore it may be identified with the linear map $\eta: V \wedge V \to V,\ \eta(v \wedge w) = \iota(vw)$. Since $A$ is a 
division algebra, this linear map $\eta$ is dissident \cite{Os62}. Any choice of a basis in $V$ identifies $\eta$ with a dissident map $\eta$ on $\R^7$, and the degree of 
$\eta$ does not depend on the chosen basis. It is therefore justified to define the {\it degree} of an 8-dimensional real quadratic division algebra $A$ by $\deg(A) := 
\deg(\eta)$. 
\begin{cor}
The degree of any 8-dimensional real quadratic division algebra is 1, 3, or 5.
\end{cor}
\begin{proof} 
Let $A \in \mathscr{D}_8^q$. Then $\deg(A) = \deg(\eta) = d$, where $1 \le d \le 5$ by \cite[Theorem 2.4]{DF06}, and $\deg(\eta)$ is odd by Proposition 2.1. 
\end{proof}

\begin{cor}
The category $\mathscr{D}_8^q$ decomposes into its non-empty full subcate\-gories $\mathscr{D}_8^{q1}, \mathscr{D}_8^{q3}$, and $\mathscr{D}_8^{q5}$.
\end{cor}

\begin{proof} 
Corollary 2.4 states that the object class of $\mathscr{D}_8^q$ is the disjoint union of the object classes of $\mathscr{D}_8^{q1}, \mathscr{D}_8^{q3}$, and 
$\mathscr{D}_8^{q5}$.

\vspace{0,1cm}
Let $f: A \to A'$ be a morphism in $\mathscr{D}_8^q$. The algebra structures on $A$ and $A'$ induce dissident maps $\eta$ and $\eta'$ on the purely imaginary hyperplanes 
$V$ and $V'$ of $A$ and $A'$ respectively. Now $f$ is injective because $A$ is a division algebra, and furthermore even bijective because $\dim(A) = \dim(A')$ is finite. 
So $f$ is an isomorphism of algebras. It induces an isomorphism of dissident maps $\sigma: \eta \to \eta'$, i.e.~a linear bijection $\sigma: V \to V'$ satisfying 
$\sigma\eta(v \wedge w) = \eta'(\sigma(v) \wedge \sigma(w))$ for all $v,w \in V$. We conclude with \cite[Proposition 3.1]{DF06} that $\deg(\eta) = \deg(\eta')$, hence 
$\deg(A) = \deg(A')$. Thus $f$ is a morphism in $\mathscr{D}_8^{qd}$ for some $d \in \{ 1,3,5 \}$. 

Altogether this proves the decomposition $\mathscr{D}_8^q = \mathscr{D}_8^{q1} \amalg \mathscr{D}_8^{q3} \amalg \mathscr{D}_8^{q5}$ of the category $\mathscr{D}_8^q$. 
Non-emptyness of its blocks $\mathscr{D}_8^{q1}, \mathscr{D}_8^{q3}$, and $\mathscr{D}_8^{q5}$ follows from \cite[Section 6]{DF06}, where objects are constructed 
for each of them. 
\end{proof}

\section{On the structure of $\mathscr{D}_8^{q1}$, $\mathscr{D}_8^{q3}$, and $\mathscr{D}_8^{q5}$}

Corollary 2.5 reduces the problem of understanding the structure of $\mathscr{D}_8^q$ to the problem of understanding the structures of $\mathscr{D}_8^{q1}, 
\mathscr{D}_8^{q3}$, and $\mathscr{D}_8^{q5}$. We proceed to summarize the present state of knowledge regarding the latter problem.

A {\it dissident triple} $(V,\xi,\eta)$ consists of a (finite-dimensional) Euclidean space $V = (V,\langle\cdot,\cdot\rangle)$, a linear form $\xi: V \wedge V \to \R$, and a 
dissident map $\eta: V \wedge V \to V$. The class of all dissident triples forms a category $\mathscr{V}$ whose morphisms $\varphi: (V,\xi,\eta) \to (V',\xi',\eta')$ are 
orthogonal linear maps $\varphi: V \to V'$ satisfying $\xi = \xi'(\varphi \wedge \varphi)$ and $\varphi\eta = \eta'(\varphi \wedge \varphi)$. If $(V,\xi,\eta)$ is a dissident 
triple, then the vector space $\mathscr{F}(V,\xi,\eta) = \R \times V$, equipped with the multiplication
\[ (\alpha,v)(\beta,w) = (\alpha\beta - \langle v,w \rangle + \xi(v \wedge w),\ \alpha w + \beta v + \eta(v \wedge w)), \]
is a real quadratic division algebra. If $\varphi: (V,\xi,\eta) \to (V',\xi',\eta')$ is a morphism of dissident triples, then the linear map
\[ \mathscr{F}(\varphi): \mathscr{F}(V,\xi,\eta) \to \mathscr{F}(V',\xi',\eta'),\ \mathscr{F}(\varphi)(\alpha,v) = (\alpha,\varphi(v)) \]
is a morphism of real quadratic division algebras. It is well-known \cite{Di00},\cite{DO02} that Osborn's theorem \cite{Os62} can be rephrased in the language of categories 
and functors as follows.
\begin{theo}
The functor $\mathscr{F}: \mathscr{V} \to \mathscr{D}^q$ is an equivalence of categories.
\end{theo}
\noindent
For each $d \in \{ 1,3,5 \}$ we denote by $\mathscr{V}_7^d$ the full subcategory of $\mathscr{V}$ formed by all dissident triples $(V,\xi,\eta)$ satisfying $\dim(V) = 7$ and 
$\deg(\eta) = d$. Then the equivalence of categories $\mathscr{F}: \mathscr{V} \to \mathscr{D}^q$ induces equivalences of categories $\mathscr{F}_7^d: \mathscr{V}_7^d \to 
\mathscr{D}_8^{qd}$ for all $d \in \{ 1,3,5 \}$. The category $\mathscr{V}_7^1$ admits an equivalent description entirely in terms of matrices, which we proceed to recall.

The octonion algebra $\O$ is well-known to be quadratic. Hence it has Frobenius decomposition $\O = \R1 \oplus V$, and thereby it determines $\R$-linear maps 
$\varrho: \O \to \R$ and $\iota: \O \to V$ such that $a = \varrho(a)1 + \iota(a)$ for all $a \in \O$. The symmetric $\R$-bilinear form 
\[ \O \times \O \to \R,\ \langle x,y \rangle = 2 \varrho(x)\varrho(y) - \frac{1}{2}\varrho(xy+yx) \]
is well-known to be positive definite, thus equipping $\O$ with the structure of a Euclidean space. Every algebra automorphism $\alpha \in {\rm Aut}(\O)$ fixes the unity 1 of 
$\O$ and is orthogonal. Since $1^\perp = V$, it induces an orthogonal linear endomorphism $\alpha_V \in {\rm O}(V)$. The map $\nu: {\rm Aut}(\O) \to {\rm O}(V),\ 
\nu(\alpha) = \alpha_V$ is an injective group homomorphism. Choosing an orthonormal basis in $V$, the subgroup $\nu({\rm Aut}(\O)) < {\rm O}(V)$ is identified with a subgroup 
of ${\rm O}(7)$, which classically is denoted by $\G_2$. Simultaneously, the dissident map 
\[ \eta: V \wedge V \to V,\ \eta(v \wedge w) = \iota(vw) \] 
is identified with a {\it vector product map} $\R^7 \wedge \R^7 \to \R^7,\ v \wedge w \mapsto v \times w$. Now denote by $\R^{7 \times 7}$ the set of all real 
$7 \times 7$-matrices, and by $\R_{{\rm ant}}^{7 \times 7},\ \R_{{\rm pds}}^{7 \times 7},\ \R_{{\rm spds}}^{7 \times 7}$ the subsets of $\R^{7 \times 7}$ consisting of all 
matrices which are antisymmetric, po\-sitive definite symmetric, and positive definite symmetric of determinant 1 respectively. We view the matrix quadruple set
\[ \mathscr{Q} = \R_{{\rm ant}}^{7 \times 7} \times \R_{{\rm ant}}^{7 \times 7} \times \R_{{\rm pds}}^{7 \times 7} \times \R_{{\rm spds}}^{7 \times 7} \]
as the object set of a groupoid $\mathscr{Q}$ whose morphisms
\[ S: (A,B,C,D) \to (A',B',C',D') \]
are the orthogonal matrices $S \in \G_2$ which satisfy
\[ (SAS^t,SBS^t,SCS^t,SDS^t) = (A',B',C',D'). \]
Then the groupoid $\mathscr{Q}$ and the category $\mathscr{V}_7^1$ are related as follows. (For proofs see \cite[Section 5]{DF06}.)
\begin{theo}
(i) If $(A,B,C,D) \in \mathscr{Q}$ then $\mathscr{G}(A,B,C,D) = \left( \R^7,\xi,\eta \right)$,\linebreak with $\xi(v \wedge w) = v^tAw$ and 
$\eta(v \wedge w) = (B+C)D(Dv \times Dw)$, is in $\mathscr{V}_7^1$.
\\[1ex]
(ii) If $S: (A,B,C,D) \to (A',B',C',D')$ is a morphism in $\mathscr{Q}$ then $\mathscr{G}(S): \mathscr{G}(A,B,C,D) \to \mathscr{G}(A',B',C',D')$, given by 
$\mathscr{G}(S)(v) = Sv$, is a morphism in $\mathscr{V}_7^1$.
\\[1ex]
(iii) The functor $\mathscr{G}: \mathscr{Q} \to \mathscr{V}_7^1$ is an equivalence of categories.
\end{theo}
\noindent
Composing the functors $\mathscr{G}$ and $\mathscr{F}_7^1$ to $\mathscr{H} = \mathscr{F}_7^1 \mathscr{G}$, we arrive at the following explicit description of the category 
$\mathscr{D}_8^{q1}$ entirely in terms of matrices.
\begin{cor}
The functor $\mathscr{H}: \mathscr{Q} \to \mathscr{D}_8^{q1}$ is an equivalence of categories. It is given on objects by $\mathscr{H}(A,B,C,D) = \R \times \R^7$, 
with multiplication $(\alpha,v)(\beta,w) = (\alpha\beta - v^tw + v^tAw, \alpha w + \beta v + (B+C)D(Dv \times Dw))$, and on morphisms by 
$\mathscr{H}(S)(\alpha,v) = (\alpha,Sv)$. 
\end{cor}
\noindent
On the other hand, we do not know any description of the categories $\mathscr{V}_7^3$ or $\mathscr{V}_7^5$ entirely in terms of matrices. Indeed, the 7-dimensional 
{\it dissident algebras} $(V,\eta)$ of degree 3 or 5 which are inherent in the objects $(V,\xi,\eta)$ of $\mathscr{V}_7^3$ or $\mathscr{V}_7^5$ respectively seem
hardly to be understood at present.

\vspace*{1cm}
\noindent
\begin{tabular}{@{}l@{\hspace{4,7cm}}l}
Ernst Dieterich & Matematiska institutionen\\
Ryszard Rubinsztein & Uppsala universitet\\
 & Box 480\\
 & SE-751 06 Uppsala\\
 & Sweden
\end{tabular}
\\[1ex]
\begin{tabular}{@{}l}
{\tt Ernst.Dieterich@math.uu.se}\\{\tt Ryszard.Rubinsztein@math.uu.se}
\end{tabular}

\end{document}